\documentclass{amsart}

\usepackage{graphicx}
\usepackage{amssymb,latexsym}
\usepackage{amsmath,amsfonts,amstext,bbm}

\newtheorem{theorem}{Theorem}
\newtheorem{lemma}{Lemma}
\newtheorem{proposition}{Proposition}

\def\XXint#1#2#3{{\setbox0=\hbox{$#1{#2#3}{\int}$ }
\vcenter{\hbox{$#2#3$ }}\kern-.6\wd0}}

\begin{document}
\title[THE FREQUENCY  FUNCTION AND ITS CONNECTIONS]
      {The frequency  function and its connections to the Lebesgue points and the Hardy-Littlewood maximal function}
      
\author{Faruk Temur}
\address{Department of Mathematics\\
         Izmir Institute of Technology, Urla, Izmir, 35430,
         Turkey}
\email{faruktemur@iyte.edu.tr}

\date{January 04, 2019}    

\begin{abstract}
The aim of this work is to extend the recent work of the author on the discrete frequency function to the more delicate  continuous frequency function $\mathcal{T}$, and further to investigate its relations to the Hardy-Littlewood maximal function $\mathcal{M}$, and to the Lebesgue points. We surmount the intricate issue of measurability of $\mathcal{T}f$ by  approaching it with a sequence of carefully constructed auxiliary functions for which measurability is easier to prove. After this we give analogues of the recent results on the discrete frequency function. We then   connect the  points of discontinuity of $\mathcal{M}f$ for $f$ simple  to  the zeros of  $\mathcal{T}f$, and to the non-Lebesgue points of $f$. 
\end{abstract}

\maketitle

\section{Introduction}\label{intro}

Let $\mathbb{R}$ be the set of  real numbers, and let $\mathbb{R}^+$ denote the set of positive real numbers. Let $f \in L^{1}(\mathbb{R})$.  We define  the average of $f$ over an  interval of radius $r\in \mathbb{R}^+$ centered at $x\in \mathbb{R}$ by
\begin{equation*}
\mathcal{A}_rf(x):= \frac{1}{2r}\int_{-r}^{r} f(x+y)dy. 
\end{equation*}
 These averages can be regarded as a function of two variables $(x,r)\in \mathbb{R}\times \mathbb{R}^+$, given by \begin{equation*}
\mathcal{A}f(x,r):=\mathcal{A}_rf(x), 
\end{equation*}
 and this gives an extension of the function $f$ to the upper half plane.  
 The  Hardy-Littlewood maximal function is then given as   
\begin{equation*}\label{dm}
\mathcal{M}f(x):=\sup_{r\in \mathbb{R}^+}\mathcal{A}_r|f|(x).
\end{equation*}

 We aim to study the distribution of the values  $r$ for which $\mathcal{M}f(x)=\mathcal{A}_r|f|(x)$. To this end we  define the sets
 \begin{equation*}
 E_{f,x}:=\{r>0:\mathcal{M}f(x)=\mathcal{A}_r|f|(x)\},
 \end{equation*}
 and  the  frequency  function 
 \begin{equation*}\label{f1}
 \mathcal{T}f(x):=\begin{cases}\inf  E_{f,x} & \text{if the set is nonempty} \\ 0 & \text{otherwise.}\end{cases}
 \end{equation*}
 Clearly this function is  well defined. Two properties emerge directly from this definition:   if the infimum of the set in the definition is greater than zero, then it belongs to the set, and if  $\mathcal{T}f(x)=0$, then there is a  sequence  of radii $\{r_n\}_{n\in \mathbb{N}}$  such that $r_n\rightarrow 0$ and   $\mathcal{A}_{r_n}|f|(x)\rightarrow  \mathcal{M}f(x)$  as $n\rightarrow \infty$. These two properties show that the two cases in the definition are intimately connected. In the next section  we will prove these two properties, and also illustrate the behavior of $\mathcal{T}f$ by calculating it  for certain functions $f$. We will observe that although the large scale behavior of $\mathcal{T}f$ is similar to the discrete case investigated in \cite{t2},  the local behavior can be much more complicated due to the possible fractal structure of $f$.   
 
 The motivation for our study of this frequency function comes from the  works \cite{ku, lu1, lu2}.  In \cite{ku} a classification  of the local maxima of $\mathcal{M}f$ based on the values of $\mathcal{T}f$ was used to great effect. In that work Kurka answered in positive a question raised by Hajlasz and Onninen in \cite{ho}: is $f \mapsto \nabla Mf$  a bounded operator from   $W^{1,p}(\mathbb{R}^n)$ to $L^1(\mathbb{R}^n)$? Indeed he    obtains the stronger result that the variation of  $\mathcal{M}f$  is at most a constant times the variation of $f$.
 In  \cite{lu1,lu2} the set $E_{f,x}$ and its variants are defined and used to prove that   $f \mapsto \nabla Mf$  is a continuous  operator on    $W^{1,p}(\mathbb{R}^n)$.  In this vein we also would like to point out the work \cite{s}, which defines a function similar to our frequency function, and using it  characterizes the sine function.     
    
Observing   the  values of $\mathcal{T}f$ is very much like  expressing a function as its Fourier series, for if  around a point  the function is more steep, then we expect $\mathcal{T}f$ to be small, if it is more dispersed then we expect $\mathcal{T}f$ to be large, and this is the exact opposite of the Fourier case. This analogy is the reason we  call $\mathcal{T}f$ the frequency function. This analogy can be seen as a part of more general and well-known connections between 
maximal functions and oscillatory integrals articulated in such works as \cite{l,st}.  We hope understanding the frequency function will contribute to the study of such connections.
Also,    as a more immediate motivation, we  hope to extract information about the Hardy-Littlewood maximal function and the  Lebesgue points using the frequency function, and our last two theorems show that this actually is possible.

Our first result is that the images under the  frequency function are measurable. This result is the key to the others, as it allows us to investigate measure theoretic properties further. It is very natural to expect the frequency function to be measurable, since it is defined using other measurable functions, however a rigorous proof turns out to be  delicate.
 \begin{theorem}
 Let  $f \in L^1(\mathbb{R})$. The function $\mathcal{T}f$ is Lebesgue measurable.
 \end{theorem}
  
 With this theorem at hand we will move to further investigation, and prove level set  estimates. As a function $f \in L^1(\mathbb{R})$ contains most of its mass in an interval of finite radius centered at the origin, for large $|x|$ 
  it is most natural to expect  $\mathcal{T}f(x)$ to be like $|x|$. If however the mass of the function is dispersed sparsely over the real line, as is the case for the function we will introduce to prove Theorem 4, the frequency function can deviate from $|x|$ for some $x$. The next section supplies further examples illustrating both of these situations. But how often can $\mathcal{T}f(x)$ deviate from $|x|$ for an arbitrary function? The next three theorems explore this issue.

\begin{theorem}
Let  $f \in L^1(\mathbb{R})$. Let $C>1$ be a  real number. Then the set
\begin{equation*}\label{sc}
\left\{x\in \mathbb{R}:
\frac{|x|}{2C}\leq\mathcal{T}f(x) \leq \frac{|x|}{C}\right\}
\end{equation*}
 is   bounded.
\end{theorem}
Our next theorem is a deeper result that gives information about the density of  level sets of a somewhat different type. We note that in this theorem and for the rest of the paper,  $|E|$ will denote the Lebesgue measure of a set $E$, and $\#E$ will denote its cardinality.

\begin{theorem}
Let  $f \in L^1(\mathbb{R})$ be a function that is not almost everywhere zero. Let $C>1$ be a real number, and  let $N\in \mathbb{N}$. Then
\begin{equation*}\label{kcn1}
\lim_{N\rightarrow \infty} \frac{ \left|\left\{x\in \mathbb{R}:|x|\leq N, \ 
\mathcal{T}f(x) \leq \frac{|x|}{C}\right\}\right|} {N}=0.
\end{equation*}
\end{theorem}

We have the following theorem that makes it clear that it is not possible to improve upon Theorem 3, even by comparing $\mathcal{T}f(x)$ to a function other than $|x|/C$.

\begin{theorem}
	For every  $\varepsilon >0$, there exists a function  $f \in L^1(\mathbb{R})$ such that 
	\begin{equation*}\label{t3}
 {| \left\{x\in \mathbb{R}:|x|\leq N, \ 
		\mathcal{T}f(x) =0\right\}|}  \geq \frac{1}{8}N / \log^{1+\varepsilon}N
	\end{equation*}
	for infinitely many values of $N\in \mathbb{N}$.
\end{theorem}

With these facts   about the frequency function at hand, we turn to its applications. 
Our next   theorem relates the points of discontinuity of $\mathcal{M}f$ to the zero set of $\mathcal{T}f$ when $f$ is a simple function. We thus extract information about the  formation of discontinuities of $\mathcal{M}f$. It may well be possible that this theorem is true for a wider class of functions, but as our arguments rely crucially on the range of $f$ being a finite set, such a result is beyond  our reach. Also as  $\mathcal{M}f, \mathcal{T}f$ are both  nonlinear operators, classical approximation by simple functions argument of measure theory does not work either.

\begin{theorem}
Let  $f \in L^1(\mathbb{R})$ be a simple function. Let $x$ be a point at which $\mathcal{M}f$ is discontinuous. Then for every $r>0$, there exist $y\in (x-r,x+r)$ such that $\mathcal{T}f(y)=0$.
\end{theorem}

Let $E\subset \mathbb{R}$. A point $x$ is called a point of density for  $E$ if 
\begin{equation*}
\lim_{r\rightarrow 0} \frac{|E\cap (x-r,x+r)|}{2r}=1.
\end{equation*} 
A point is called an  exceptional point for $E$ if it is not a density point for either $E$ or its complement $E^c$. It is well known that if both $E,E^c$  have nonzero measure then $E$ has an exceptional point, see \cite{gr,ku2,sz}.

A more general concept than that of density points is the concept of Lebesgue points. For  $f\in L^1_{loc}(\mathbb{R})$ we call $x$  a Lebesgue point of $f$ if there exist $c(x)$ such that 
\begin{equation*}\label{lp}
\lim_{r\rightarrow 0} \frac{1}{2r}\int_{-r}^{r}|f(x+t)-c(x)|dt=0.
\end{equation*}
It is well known that for almost every $x$ this equation  is satisfied with $c(x)=f(x)$, see \cite{fl,lin}. We combine the existence of exceptional points with Theorem 5 to prove the first part of the  following theorem. As opposed to  Theorem 5, the phenomenon observed in this  first part seems to be peculiar to simple functions, a relatively easy example showing this will be provided. Then we use topological arguments to prove the other direction of the theorem. This other direction is not true even for simple functions, as we will show.

\begin{theorem}
Let $f\in L^1(\mathbb{R})$ be a simple function. If $x$ is a point of discontinuity for $\mathcal{M}f$, then for every $r>0$ there exist $y\in (x-r,x+r)$ such that $y$ is not a Lebesgue point of $f$. Conversely let $f\in L^1(\mathbb{R})$ be a  characteristic function.  If $x$ is a non-Lebesgue point for $f$, then for every $r>0$ there exist $y\in (x-r,x+r)$ such that $\mathcal{M}f$ is discontinuous at $y$.
\end{theorem}

Even for characteristic functions it is possible for a point of discontinuity of $\mathcal{M}f$ to be a Lebesgue point of $f$, and  $\mathcal{M}f$ to be continuous at a non-Lebesgue point of $f$.  Examples of both  situations will be furnished. Therefore it is not possible to improve the theorem in this direction either.  

In the next section we  include examples to show that the operator $ \mathcal{T}$ is indeed very rough, and a small change of the function $f$ can lead to great changes in  $\mathcal{T}f$. Also included are functions $f$ for which a small change of $x$ can lead to great changes in  $\mathcal{T}f(x)$.   We will thus  gain insight into  the regularity properties of the frequency function. After these examples we will  prove the basic properties of the frequency function mentioned right after its definition. The rest of the article is devoted to proofs of our theorems: in the third section we will prove our first   theorem. To this end we will introduce certain auxiliary functions and investigate their properties. In the fourth section we will prove the next three theorems that concern the size of the frequency function.    The final section is reserved for  the last theorems.

As a last remark we note  that both the maximal function and the frequency function do not distinguish between a function and its absolute value, and further, the   Lebesgue points of a function  are  also   Lebesgue points for its absolute value. Therefore it suffices to prove all our results in this work only for nonnegative functions. Also if two functions are the same  almost everywhere, then    their images under both the maximal function and the frequency function and their Lebesgue points are the same. Thus if one of our results holds for one of these functions, it also holds for the other.


\section{Examples and certain basic properties}

\subsection{Examples}
We will now compute the frequency function $\mathcal{T}f$ for certain functions $f$ to get a sense of its behavior. We start with considering the zero function: let  $f_1(x)=0$ for almost every $x$. Clearly $ \mathcal{M}f_1(x)=\mathcal{T}f_1(x)=0$ for every $x$. Notice that both the maximal function and the frequency function remove whatever irregularities may occur due to the behavior on measure zero sets.  As our Theorem 3 shows this is the only function for which  the image under the frequency function is not  mostly greater than or similar to $|x|$.  

  We now consider the function $f_2(x)=\chi_{[-1,1]}(x)$. We have 
 \begin{equation*}
 \mathcal{M}f_2(x):= \begin{cases} 
 1 &  |x|<1 \\
 \frac{1}{|x|+1}& |x| \geq 1,
 \end{cases} \ \ \ \ \ \ \ \ \ \ \ \ \
 \mathcal{T}f_2(x):= \begin{cases} 
 0 & |x|\leq 1 \\
 {|x|+1}& |x| > 1.
 \end{cases}
 \end{equation*}
 As is clear for $|x|$ large, $\mathcal{T}f_2(x)$ is always like $|x|$. 
 
 In our third example we demonstrate that for some functions this is not so.
  \begin{equation*}
 f_3(x):= \begin{cases} 
  \frac{1}{n^2} &  2^n\leq x \leq 2^n+1, \  n\in \mathbb{N} \\
  0&  \text{elsewhere.}
  \end{cases}
  \end{equation*}
 Clearly this function is integrable, and yet $ \mathcal{T}f_3(x)=0$ whenever $2^n< |x| <2^n+1, $ for $n$ large. This is simply because the function $f_3(x)$ gets sparser and smaller as $x$ gets larger. We thus see that there can be points  $x$ arbitrarily distant from the origin for which $\mathcal{T}f_3(x)$ is not comparable to $|x|.$ Our  Theorem 3 says that nonetheless the density of such points decrease to zero.

 Our fourth example is actually a sequence of examples 
 \[f_{4,k}(x)=\frac{1}{k}\chi_{[-1,1]}(x),\]
 where $k\in \mathbb{N} $. Observe that 
  \begin{equation*}
  \mathcal{M}f_{4,k}(x):= \begin{cases} 
  1/k &  |x|<1 \\
  \frac{1}{k|x|+k}& |x| \geq 1,
  \end{cases} \ \ \ \ \ \ \ \ \ \ \ \ \
  \mathcal{T}f_{4,k}(x):= \begin{cases} 
  0 & |x|\leq 1 \\
  {|x|+1}& |x| > 1.
  \end{cases}
  \end{equation*}
 Notice that whereas $ \mathcal{M}f_{4,k}$ does depend on $k$,  $\mathcal{T}f_{4,k}$ is independent of it. As $k\rightarrow \infty$,  the examples $f_{4,k}$ converge to the zero function both pointwise and in $L^1$ sense, yet $ \mathcal{T}f_{4,k}$  never changes. This shows that even when two functions are close to each other pointwise or in $L^1$ sense their frequency functions can be very different.  
 
 Our fifth example is of fractal type. Let 
 \begin{equation*}
 f_5(x):=\chi_{(-1,0)}(x)+\sum_{n\in \mathbb{N}}\chi_{(2^{-n+1}-2^{-n-1},2^{-n+1})}(x).
 \end{equation*}
  This function is the characteristic function of  an open set. If $x$ is in this set then $\mathcal{T}f_{5}(x)=0$. But if $x=2^{-n+1}-2^{-n-1}$  for $n\geq 2$, then $\mathcal{T}f_{5}(x)=1-x$. Thus we see that on the interval $(0,\epsilon)$ the function $\mathcal{T}f_{5}(x)$  switches from 0 to $1-x$ infinitely often. This shows that the behavior of the image under the  frequency function can be highly irregular even on arbitrarily small intervals, and that it can show fractal type behavior.

 \subsection{Basic properties of the frequency function} Here we will demonstrate the two properties of the frequency function mentioned right after its definition. We first observe that if $f \in L^1(\mathbb{R})$
then
 \begin{equation}\label{bd}
 |\mathcal{A}_rf(x)|\leq \|f\|_1/2r, \quad \quad \text{and} \quad \quad  \lim_{r\rightarrow  \infty} \mathcal{A}_rf(x)= 0. 
 \end{equation}
Next  we introduce a well known result that will be repeatedly used in the rest of this work.  This is   Lemma 3.16 in  \cite{fl}, and a proof can be found there. 
\begin{lemma}
 Let  $f \in L^1(\mathbb{R})$. Then the function $\mathcal{A}f(x,r):\mathbb{R}\times \mathbb{R}^+\rightarrow \mathbb{C}$ is continuous.
\end{lemma}

We are now ready to obtain our properties. The proofs utilize \eqref{bd}, Lemma 1 and  the least upper bound property of $\mathbb{R}$.

\begin{proposition}
Let $f\in L^1(\mathbb{R})$.  If $\mathcal{T}f(x)>0$ then  $\mathcal{T}f(x)\in \{r>0:\mathcal{M}f(x)=\mathcal{A}_r|f|(x)\}$. 
\end{proposition}

\begin{proof}
This means   $ E_{f,x}$ is nonempty and   $\mathcal{T}f(x)=\inf  E_{f,x} $, so there exist a sequence  $\{r_n\}_{n\in \mathbb{N}}$ such that $\mathcal{T}f(x)\leq r_n\leq \mathcal{T}f(x) + n^{-1}$, and $\mathcal{M}f(x)=\mathcal{A}_{r_n}|f|(x)=\mathcal{A}|f|(x,r_n)$. Then letting $n\rightarrow \infty $ and applying Lemma 1 completes the proof.
\end{proof}

\begin{proposition}
Let $f\in L^1(\mathbb{R})$.  If $\mathcal{T}f(x)=0$ then there exist a sequence $\{r_n\}_{n\in \mathbb{N}}$ such that $r_n \rightarrow 0$ and $\mathcal{A}_{r_n}|f|(x) \rightarrow  \mathcal{M}f(x)$ as $n\rightarrow \infty$. 
\end{proposition}

\begin{proof}
 This is clear if $f=0$ almost everywhere, or if $ E_{f,x}$ is nonempty. We thus assume otherwise, in which case $\mathcal{M}f(x)$ is either a positive real number or infinite  for any $x\in \mathbb{R}$. If it is infinite,  there must be  a sequence $\{r_n\}_{n\in \mathbb{N}}$ with $ n\|f\|_1\leq \mathcal{A}_{r_n}|f|(x).$ But then by \eqref{bd} we have $ n\|f\|_1\leq \|f\|_1/2r_n $, which in turn yields
 $r_n\leq n^{-1}$. Thus $\{r_n\}_{n\in \mathbb{N}}$ is  a sequence with desired properties.

  If, on the other hand, $Mf(x)$ is a positive real number,  we must  have  values $r_k$ such that $ (1-2^{-k})\mathcal{M}f(x)\leq \mathcal{A}_{r_k}|f|(x) \leq \|f\|_1/2r_k$, with the last inequality coming from \eqref{bd}.
This yields $ r_k \leq {\|f\|_1}/ { \mathcal{M}f(x)}.$
Then by the Bolzano-Weierstrass theorem we must have a convergent subsequence $\{r_{k_n}\}_{n\in \mathbb{N}}$, with  limit  $r$ in $[0,\|f\|_1/\mathcal{M}f(x)].$ This $r$ cannot be positive, for in that case by Lemma 1 we have $\mathcal{A}_{r}|f|(x)=\mathcal{M}f(x)$, which is a contradiction. Therefore this subsequence converges to 0, and we have $\mathcal{A}_{r_{k_n}}|f|(x) \rightarrow  \mathcal{M}f(x)$ as $n\rightarrow \infty.$

\end{proof}

     
     \section{The measurability of the frequency function}
     
     In this section we prove Theorem 1.  We recall that the measurability of $\mathcal{M}f$ follows easily from the  continuity of averages. Indeed
     \begin{equation*}
     \mathcal{M}f^{-1}((\alpha,\infty])=\bigcup_{r>0} \mathcal{A}_r|f|^{-1}((\alpha,\infty)),
     \end{equation*}
     and since $\mathcal{A}_r|f|$ are continuous, $ \mathcal{M}f^{-1}((\alpha,\infty])$ is not only measurable but also open. Therefore $\mathcal{M}f$ is not only measurable but also lower semi-continuous. Unfortunately  arguments of this type are not available for the frequency function. Indeed, examples $f_2$ and $f_5$ of section 2 clearly demonstrate that $\mathcal{T}f$ need not be lower or upper semicontinuous. We therefore need a different method. We  will write the frequency function as the  limit of a sequence of auxiliary functions for which measurability can be  proved using countability arguments.
     Let $f\in L^1(\mathbb{R})$. We define for  $k,l\in \mathbb{N}$ the sets
      \begin{equation*}
      E_{f,x,k,l}:=\{r\in \mathbb{Q}^+: r\geq 2^{-l}, \ \ \mathcal{A}_r|f|(x) +2^{-k}  \geq \mathcal{M}f(x)\},
      \end{equation*}
       and the operators
\begin{equation*}
\mathcal{T}_{k,l}f(x):=\begin{cases} \inf E_{f,x,k,l}  & \text{if the set is nonempty} \\ 0 & \text{else.}  \end{cases}
\end{equation*}
  Clearly  $\mathcal{T}_{k,l}f:\mathbb{R}\rightarrow \mathbb{R}$  are well defined. If $E_{f,x,k,l}$ is nonempty, then its infimum $\mathcal{T}_{k,l}f(x)$ may not be rational, but still satisfies  
  \begin{equation}\label{bd2}
   \mathcal{T}_{k,l}f(x)\geq 2^{-l}, \quad \text{and}  \quad  \mathcal{A}_{\mathcal{T}_{k,l}f(x)}|f|(x) +2^{-k}  \geq \mathcal{M}f(x).
  \end{equation}
  If $f$ is not zero almost everywhere, then $ E_{f,x,k,l}$ are bounded for $k$ large enough. Indeed if we pick $k$ such that $2^{-k}<\mathcal{M}f(x)/2$, then by \eqref{bd}   for  $r\in  E_{f,x,k,l}$ we have
  \begin{equation}\label{bd3}
     r\leq \|f\|_1/\mathcal{M}f(x).
    \end{equation}
    The next proposition will exploit the countability of $E_{f,x,k,l}$ to  show that  $\mathcal{T}_{k,l}f$ are  measurable.

\begin{proposition}
 Let $f\in L^1(\mathbb{R})$. The function $\mathcal{T}_{k,l}f$ is measurable for every $k,l\in \mathbb{N}$.
\end{proposition}  
  
\begin{proof}
We fix  $k,l\in \mathbb{N}$. It suffices to show that for any $\alpha$ the set $\mathcal{T}_{k,l}f^{-1}([\alpha,\infty))$ is measurable. When $\alpha < 2^{-l}$  this set  is either $\mathbb{R}$ or $\mathcal{T}_{k,l}f^{-1}([2^{-l},\infty)),$ therefore we assume that $\alpha \geq 2^{-l}.$   We
consider for every $r\in \mathbb{Q}^+$ the set
\begin{equation*}
S_r:=\{x\in \mathbb{R}:     \mathcal{A}_r|f|(x) +2^{-k}  \geq \mathcal{M}f(x)    \}.
\end{equation*}
Clearly these sets are measurable. We claim that 
\begin{equation*}
\mathcal{T}_{k,l}f^{-1}([2^{-l}\infty)) =\bigcup_{r\in [2^{-l},\infty)\cap\mathbb{Q} } S_r,
\end{equation*}
and for $\alpha>2^{-l}$ that
\begin{equation*}
\mathcal{T}_{k,l}f^{-1}([\alpha,\infty)) =\Big(\bigcup_{r\in [\alpha,\infty)\cap\mathbb{Q} } S_r\Big)\setminus \Big( \bigcup_{s\in [2^{-l},\alpha)\cap\mathbb{Q} } S_s \Big).
\end{equation*}
We will verify the second claim, the first follows from the same arguments with even less difficulty.  Let $x\in \mathcal{T}_{k,l}f^{-1}([\alpha,\infty))$. So   there must be a rational $r\geq  \mathcal{T}_{k,l}f(x)\geq \alpha$ with  $\mathcal{A}_r|f|(x) +2^{-k}  \geq \mathcal{M}f(x)$, and there can be no rational  $2^{-l}\leq s<\alpha$ with $\mathcal{A}_s|f|(x) +2^{-k}  \geq \mathcal{M}f(x)$. This proves inclusion in one direction.
Conversely let $x$ be in the set on the right hand side. Then $x\in S_r$  for some $r\in [\alpha,\infty)\cap\mathbb{Q},$ and $x\notin S_s$  for any $[2^{-l},\alpha)\cap\mathbb{Q}.$ Thus $E_{f,x,k,l}$ is not empty, but can contain no element $2^{-l}\leq s<\alpha$, and this  means  $\mathcal{T}_{k,l}f(x)=\inf E_{f,x,k,l}\geq \alpha$. This concludes the proof.

\end{proof}     
     
 We now are ready to  show the measurability of $\mathcal{T}f$ by writing it as a limit. For each fixed  $l$, we prove that   $\lim_{k\rightarrow \infty}\mathcal{T}_{k,l}f$ converge to   a real valued  function $\mathcal{T}_{l}f$. We then show that $\lim_{l\rightarrow \infty}\mathcal{T}_{l}f=\mathcal{T}f.$

\begin{proof}
 If $f=0$ almost everywhere, then  $\mathcal{T}_{k,l}f=2^{-l}$, and  therefore both the existence of $\mathcal{T}_{l}f$ and their convergence to $\mathcal{T}f$ are clear. We therefore assume otherwise. 
 
 We first concentrate on the  existence of $\mathcal{T}_{l}f$, and therefore fix   $l\in \mathbb{N}$. We observe that for any $x$ we have 
  \begin{equation}\label{bd4}
  E_{f,x,1,l}\supseteq E_{f,x,2,l}\supseteq E_{f,x,3,l}\supseteq \ldots
  \end{equation}   
   If for $x$,  there exists $k_x\in \mathbb{N}$ such that $E_{f,x,k_x,l}$ is empty, then so is $E_{f,x,k,l}$ for all $k\geq k_x$, and therefore $\mathcal{T}_{k,l}f(x)=0$ for such $k$. In this case  $\lim_{k\rightarrow \infty}{T}_{k,l}f(x)=0.$ If for $x$ all of $E_{f,x,k,l}$ are nonempty, then $\mathcal{T}_{k,l}f(x)=\inf E_{f,x,k,l}$, and we have the chains 
\begin{equation*}
  \inf E_{f,x,1,l}\leq \inf E_{f,x,2,l}\leq \inf E_{f,x,3,l}\leq  \ldots
  \end{equation*}
\begin{equation*}
\mathcal{T}_{1,l}f(x) \leq\mathcal{T}_{2,l}f(x)\leq \mathcal{T}_{3,l}f(x)\leq  \ldots
  \end{equation*}
 Thus  $\lim_{l\rightarrow \infty}\mathcal{T}_{k,l}f(x)$ exists. Moreover, since by   \eqref{bd3} for large $k$  the sets $E_{f,x,k,l}$ are bounded above by a common bound,  this limit is finite. 

Hence we can define  a real valued measurable function 
  $\mathcal{T}_{l}f(x):=\lim_{k\rightarrow \infty}\mathcal{T}_{k,l}f(x)$, and  
  reduce to proving 
  $\lim_{l\rightarrow \infty}\mathcal{T}_{l}f(x)= \mathcal{T}f(x)$ for every real $x$.
This we will do  in cases. For any $x$ exactly one of the following is true:\\
{\bf I.} The set $E_{f,x}$ is empty, and thus $\mathcal{T}f(x)=0$.\\
{\bf II.} The set $E_{f,x}$ is nonempty with $\mathcal{T}f(x)=\inf E_{f,x} >0 $.\\
{\bf III.} The set $E_{f,x}$ is nonempty with $\mathcal{T}f(x)=\inf E_{f,x} =0 $.

 {\bf  Case I.} We will see that for such $x$ the sets in the chain \eqref{bd4} are empty after some point, and therefore for any $l\in\mathbb{N}$ we have $\mathcal{T}_{l}f(x)=0$. We fix $l\in\mathbb{N}.$  Owing to Lemma 1 and  \eqref{bd} the function  $\mathcal{A}|f|(x,r)$ attains its supremum on $\{x\}\times [2^{-l},\infty)$, at some $(x,r_x)$. 
 As $E_{f,x}$ is empty we have  $\mathcal{A}_{r_x}|f|(x)<\mathcal{M}f(x).$  
 For  any $k$ with $2^{-k}<\mathcal{M}f(x)-\mathcal{A}_{r_x}|f|(x)$ the set $E_{f,x,k,l}$ is empty, and we are done. 

 {\bf  Case II.} We will see that $\mathcal{T}_{l}f(x)=\mathcal{T}f(x)$ for any $l$ with $2^{-l}<\mathcal{T}f(x)/2$. Fix one such $l$. By Proposition 1  we have $\mathcal{T}f(x)\in  E_{f,x}$.  Then by Lemma 1, the sets    $E_{f,x,k,l}$ are nonempty for every $k\in \mathbb{N}$, and  actually contain elements smaller than $\mathcal{T}f(x)$ for every $k$.  Thus $\mathcal{T}_{k,l}f(x)=\inf E_{f,x,k,l}\leq \mathcal{T}f(x)$, and therefore $2^{-l}\leq \mathcal{T}_{l}f(x)\leq \mathcal{T}f(x)$. But  \eqref{bd2} together with 
 Lemma 1 yields  $\mathcal{A}_{\mathcal{T}_{l}f(x)}|f|(x)  \geq \mathcal{M}f(x)$ if we take the limit $k\rightarrow \infty.$ This 
 requires $\mathcal{T}_lf(x)\geq \mathcal{T}f(x)$ and we are done.

  {\bf Case III.} Since $E_{f,x}$ is nonempty but $\inf E_{f,x}=0$, there must be a sequence $\{r_n\}_{n\in \mathbb{N}}\subseteq E_{f,x}$   converging to zero. Fix $n\in \mathbb{N}$, and let  $l$ be such that $2^{-l}\leq r_n/2$. By Lemma 1 for each $k$, the set  $E_{f,x,k,l}$ must contain an  element $r_k<r_n$. Thus $E_{f,x,k,l}$ are nonempty, and $\mathcal{T}_{k,l}f(x)=\inf E_{f,x,k,l}\leq r_n$. Taking limits we deduce that $\mathcal{T}_{l}f(x)\leq r_n,$  which in turn leads to   $\limsup_{l\rightarrow \infty}\mathcal{T}_{l}f(x)\leq r_n$.
Letting $n\rightarrow \infty$, this means $\limsup_{l\rightarrow \infty}\mathcal{T}_{l}f(x)=0.$
Therefore $\lim_{l\rightarrow \infty}\mathcal{T}_{l}f(x)=0.$

\end{proof}

 \section{ The size of the frequency function}
 \subsection{Proof of Theorem 2}
 We use the same idea as in the proof of the analogous theorem in \cite{t2}, that  if the set was unbounded, we could extract a sequence of points around each of which the integral of $|f|$ would be comparable to $\|f\|_1.$ The analogous result in \cite{t2} proves finiteness, while here finiteness is  wrong, and we have to make do with boundedness.
 
  \begin{proof}  The result is clear if $f$ is almost everywhere  zero. Therefore we will assume otherwise. Assume to the contrary that the set is not bounded. Then at least one of $S_C\cap \mathbb{R}^+, S_C\cap \mathbb{R}^-$ must be unbounded: we will assume $S_C\cap \mathbb{R}^+$ is unbounded, the other case follows from the same arguments.   Let 
 \[A:=\frac{C+1}{C-1},\ \ \ \ \ B:=\frac{C+1}{C},\ \ \ \ \ D:=\frac{C-1}{C}. \]
 Since $f\in L^1(\mathbb{R})$ there must be some $m\in \mathbb{N}$ with  
 \[\int_{-m}^m |f(x)|dx\geq \frac{\|f\|_1}{2}.\]
 Owing to our unboundedness assumption on $S_C\cap \mathbb{R}^+$  we can find  an element  $ x_1\in S_C$ with $x_1>m$. Again by the same assumption there exists  $x_2\in S_C$ with  $x_{2}>2Ax_1$. Proceeding thus 
 we extract a sequence   $\{x_i\}_{i\in \mathbb{N}}\subseteq S_C$ with   $x_{i+1}>2Ax_i$ for  each natural number $i$. Then from Proposition 1 we have 
\begin{equation*}
\begin{aligned}
\mathcal{M}f(x_i) =\mathcal{A}_{\mathcal{T}f(x_i)}|f|(x_i)  \leq \frac{1}{2\mathcal{T}f(x_i)}\int_{ [Dx_i,Bx_i]}|f(x)|dx.
\end{aligned}
\end{equation*} 
This implies 
\[\frac{x_i}{C}\cdot \mathcal{M}f(x_i) \leq \int_{ [Dx_i,Bx_i]}|f(x)|dx.\] 
We observe that since $A=B/D$, we have $Dx_{i+1}>2Bx_i$, and thus the intervals $[Dx_i,Bx_i]$ never intersect. Hence we must have
\begin{equation}\label{t1}
 \sum_{i\in \mathbb{N}} \frac{x_i}{C}\cdot \mathcal{M}f(x_i) \leq \sum_{i\in \mathbb{N}} \int_{ [Dx_i,Bx_i]}|f(x)|dx\leq \|f\|_1.
\end{equation}
On the other hand, as 
 $x_i> m$ we  have
\[\mathcal{M}f(x_i) \geq A_{2x_i}|f|(x_i) = \frac{1}{4x_i}\int_{-2x_i}^{2x_i}|f(x_i+x)|dx=\frac{1}{4x_i}\int_{-x_i}^{3x_i}|f(x)|dx\geq \frac{\|f\|_1}{8x_i} .\]
  Then from \eqref{t1} we obtain the contradiction
 \[\sum_{i\in \mathbb{N}} \frac{\|f\|_1}{8C}\leq \|f\|_1.\]
Therefore  $S_C\cap \mathbb{R}^+$ must be bounded.
\end{proof}

\subsection{Proof of Theorem 3} 
The proof of Theorem 3 uses the ideas  introduced in its analogue in \cite{t2}, but also accounts for the difference that now the maximal function may be infinite for some points in the domain. We will again use the Vitali covering lemma, which we state  below.

\begin{lemma}[Vitali Covering Lemma]
Let $\{B_i\}_{i=1}^m$ be a finite collection of open intervals  with  finite length. Let $E\subseteq \mathbb{R}$ be a subset    covered by these intervals. Then we can find a disjoint subcollection  $\{B_{i_k}\}_{k=1}^n$ of $\{B_i\}_{i=1}^m$ such that 
\[\sum_{k=1}^n |B_{i_k}|\geq \frac{|E|}{3}.\]
\end{lemma}
 A proof can be found in \cite{ru}.
We   now prove  Theorem 3.

\begin{proof}  The classical weak boundedness result for the maximal function states
\begin{equation*}
|\{x:\mathcal{M}f(x)>\lambda    \}|  \leq \frac{3}{\lambda}\|f\|_1,
\end{equation*}
and this implies the set $S_{\infty}$ of points $x$ where $\mathcal{M}f(x)=\infty$ has zero Lebesgue measure.
Therefore in any set of positive measure we can find points at which  $\mathcal{M}f$ is finite. 

We define  $A,B,D$  exactly as in the proof of Theorem 2. We will use the notations
	\begin{equation*}
	K_N=\left\{x\in \mathbb{R}:|x|\leq N, \ 
	\mathcal{T}f(x) \leq \frac{|x|}{C}\right\}\setminus S_{\infty}, \ \ \ K_{N}^+=K_N\cap \mathbb{R}^+, \ \ \ K_{N}^-=K_N\cap \mathbb{R}^-.
	\end{equation*}
	 We will prove  
\begin{equation}\label{kn}
\lim_{N\rightarrow \infty} \frac{  |K_{N}^+|}{N}= 0,
\end{equation} 
 and  the same arguments give the analogous  result for $K_{N}^-$. The theorem clearly follows from  these two results.

 We assume to the contrary that \eqref{kn} is wrong, that 
 there exists a small $\epsilon>0$ such that 
$| K_{N_i}^+|/N_i\geq \epsilon$
for a strictly increasing  sequence $\{N_i\}_{i\in\mathbb{N} }\subseteq \mathbb{N}$.  We pick a natural number $M$   such that 
 \[\int_{-M}^M |f(x)|dx\geq \frac{\|f\|_1}{2},\]
and $M>10^{10^{10A\epsilon^{-10}}}.$
We extract a subsequence from  $\{N_i\}_{i\in\mathbb{N} }$ as follows. Let $N_{i_1}$  be such that $N_{i_1}\geq M$, and  let $N_{i_{k+1}}\geq 10A \epsilon^{-1}N_{i_{k}}$ for every $k\geq 1$. We fix $k\geq 1$. We have
\[ |K_{N_{i_{2k}}}^+\setminus K_{N_{i_{2k-1}}}^+|\geq \frac{9\epsilon N_{i_{2k}}}{10}\geq 9N_{i_{2k-1}}.\]
 For $x\in K_{N_{i_{2k}}}^+\setminus K_{N_{i_{2k-1}}}^+ $ we can find positive real numbers $r_x\leq x/C$ satisfying $\mathcal{A}_{r_x}|f|(x)\geq \mathcal{M}f(x)/2$, by taking $r_x=\mathcal{T}f(x)$ if $\mathcal{T}f(x)$ is positive, and by using Proposition 2 if $\mathcal{T}f(x)=0$ . Also, since $x>M$ 
\[\mathcal{M}f(x)\geq \mathcal{A}_{2x}|f|(x)=\frac{1}{4x}\int_{-2x}^{2x}|f(x+t)|dt =\frac{1}{4x}\int_{-x}^{3x}|f(t)|dt\geq \frac{\|f\|_1 }{8x}. \]
We combine these two to  obtain 
\begin{equation}\label{fr} 2r_x\mathcal{A}_{r_x}|f|(x)=\int_{-r_x}^{r_x}|f(x+t)|dt\geq \frac{r_x }{8x}\|f\|_1.
 \end{equation}
The intervals $(x-r_x,x+r_x)$ cover the set $K_{N_{i_{2k}}}^+\setminus K_{N_{i_{2k-1}}}^+$. By the inner regularity of the Lebesgue measure we can find   a compact subset $K$ of this set with at least half its measure, and  there exists a finite subcover of $K$ consisting of intervals $(x-r_x,x+r_x)$.    By the Vitali covering lemma we have a subset $x_1,x_2,\ldots x_{p_k}$ for which the intervals $(x_i-r_{x_i},x_i+r_{x_i}), \ 1\leq i\leq p_k$ are disjoint, and
\[\sum_{i=1}^{p_k}2r_{x_i} \geq \frac{1}{3}|K|\geq\frac{1}{6}| K_{N_{i_{2k}}}^+\setminus K_{N_{i_{2k-1}}}^+|\geq \frac{9\epsilon N_{i_{2k}}}{60}. \]
Combining with \eqref{fr} yields
\begin{equation*}
\begin{aligned}
\sum_{i=1}^{p_k} \int_{-r_{x_i}}^{r_{x_i}}|f(x_i+t)|dt &\geq \sum_{i=1}^{p_k}\frac{r_{x_i} }{8x_i}\|f\|_1\\
&\geq \frac{\|f\|_1 }{8N_{i_{2k}}}\sum_{i=1}^{p_k}r_{x_i}\\
&\geq \frac{\|f\|_1 }{8N_{i_{2k}}}\frac{9\epsilon N_{i_{2k}}}{120}\\
&\geq \frac{ \epsilon \|f\|_1}{120}.
\end{aligned}
\end{equation*}
But  the intervals $(x_i-r_{x_i},x_i+r_{x_i})$ are disjoint, therefore we  have
\begin{equation*}
\int_{DN_{i_{2k-1}}}^{BN_{i_{2k}}}|f(t)|dt\geq \sum_{i=1}^{p_k} \int_{-r_{x_i}}^{r_{x_i}}|f(x_i+t)|dt \geq 
 \frac{ \epsilon \|f\|_1}{120}.
\end{equation*}
As $[DN_{i_{2k-1}},BN_{i_{2k}}]$ are disjoint for each natural number $k$, summing over $k$ we have
\begin{equation*}
\|f\|_1\geq \sum_{k\in \mathbb{N}} \ \ \int_{DN_{i_{2k-1}}}^{BN_{i_{2k}}}|f(t)|dt \geq 
\sum_{k\in \mathbb{N}} \frac{ \epsilon \|f\|_1}{120}
\end{equation*}
which is a contradiction.

\end{proof}

\subsection{Proof of Theorem 4}
The function we provide is analogous to the one in \cite{t2}.  We make use of the sparsity of the function to show that the frequency function vanishes.  
\begin{proof} We may assume  $\varepsilon$ is much smaller than 1. Let  $\mathcal{A}_{r,s}f$ denote the average of $f$ over $[r,s]$. For any integer $m\geq 10$ we denote  $m':=m\log^{1+\varepsilon/2} m$, and  $m'':=m\log^{1+\varepsilon} m.$  We define 
\[f(x):=\sum_{m=10}^{\infty} \frac{1}{m'} \chi_{ ( m'', m''+1)}(x). \]

Let $M>10^{10^{10\varepsilon^{-10}}}$ be a natural number, and let $N$ be the smallest integer not less than $M''$. Consider $m\in [M/2,M]\cap \mathbb{N}$ and values  $x\in (m'', m''+1) $. For $\delta>0$ small enough    we of course have
 $\mathcal{A}_{\delta}f(x)=f(x)=1/m'.$
We will show that $\mathcal{A}_rf(x)$ cannot be larger than this for any $r$. We have
 $\|f\|_1=C_{\varepsilon}\leq 2/\varepsilon$, 
therefore, if $r\geq x-10$, then   \eqref{bd} leads to  
$\mathcal{A}_rf(x)\leq {\|f\|_1}/{x}\leq {2/\varepsilon m''}.$
Considering our choice of $M$, this is less than $1/m'$. Clearly the case $r\leq 10$ is  also impossible. Thus remains the case $10<r<x-10$. In this case observe that
\begin{equation*}
\begin{aligned}
\mathcal{A}_rf(x)=\frac{1}{2r}\int_{-r}^rf(x+t)dt \leq \frac{1}{r}\int_{m''-r}^{m''+1}f(t)dt\leq \frac{1}{10m'}+ {2}\mathcal{A}_{m''-r,m''}f
\end{aligned}
\end{equation*}
 A moments consideration makes it clear that to maximize the last expression it is most advantageous to  choose $r$ such that $m''-r=k''$ for some $10\leq k<m$. Furthermore it is best to choose $k=10$, for the nonzero values of $f$ get smaller and also sparser. Thus the last average is not greater than $\mathcal{A}_{10'',m''}f$, which in turn  is bounded by $ 4/\varepsilon m''.$
Therefore we can conclude that $\mathcal{A}_{\delta}f(x)=\mathcal{M}f(x)$ for any $\delta $ small enough, and  $\mathcal{T}f(x)=0.$  
Hence we have
\[{|\left\{x:|x|\leq N, \ 
		\mathcal{T}f(x) =0\right\}|}\geq \frac{M}{4}\geq \frac{1}{8}N /\log^{1+\varepsilon}N.\]

\end{proof}


\section{Connections and applications}

In this section we prove our last two theorems and thereby establish connections of the frequency function with various other concepts of harmonic analysis.

\subsection{Proof of Theorem 5}
The proof relies crucially of the range of $f$ being finite. Starting from the highest value $f$ takes, which since we may assume $f$ to be nonnegative makes sense, we iterate two arguments: that if  the average of $f$  over a set is equal to its maximum on that set,  then $f$  must have that value almost everywhere on that set, and that a discontinuity of $\mathcal{M}f$ cannot be approached by a sequence of points over  which $\mathcal{M}f$ is greater than its value at the discontinuity and $\mathcal{T}f$ is bounded below by a positive constant. While the first of these arguments is clear, the second requires a more rigorous expression which we provide below.

\begin{lemma}
Let $f\in L^1(\mathbb{R})\cap  L^{\infty}(\mathbb{R})$, and let $\epsilon>0$.  If $\{x_n\}_{n\in \mathbb{N}}$ is a sequence converging to $x$ with $\mathcal{M}f(x_n)\geq \mathcal{M}f(x)+\epsilon$ for all $n\in \mathbb{N}$, then the set $\{\mathcal{T}f(x_n): n\in \mathbb{N}\}$  cannot be bounded below by a positive constant.
\end{lemma}

\begin{proof}
Suppose $r>0$ is such a lower bound. By Proposition 1, and \eqref{bd}, for each $n\in \mathbb{N}$ we have  
\begin{equation*}
\epsilon \leq \mathcal{M}f(x_n)=\mathcal{A}_{\mathcal{T}f(x_n)}|f|(x_n)\leq {\|f\|_1}/{2\mathcal{T}f(x_n)},
\end{equation*}
 which  implies $\mathcal{T}f(x_n)\in [r,\|f\|_1/2\epsilon]$ for each $n\in \mathbb{N}$. By the Bolzano-Weierstrass theorem there is a subsequence $\{x_{n_k}\}_{k\in \mathbb{N}}$ for which $\mathcal{T}f(x_{n_k})$ converges to $r'\in [r,\|f\|_1/2\epsilon].$
Then by Lemma 1 we have as $k\rightarrow \infty$
\begin{equation*}
 \mathcal{M}f(x)+\epsilon \leq  \mathcal{M}f(x_{n_k})=\mathcal{A}_{\mathcal{T}f(x_{n_k})}|f|(x_{n_k})\rightarrow \mathcal{A}_{r'}|f|(x)\leq \mathcal{M}f(x),
\end{equation*}
a contradiction.
\end{proof}

We now present the proof of our theorem.

\begin{proof}
If $f$ is zero almost everywhere, then $\mathcal{M}f$ is never discontinuous, therefore we assume otherwise.  Owing to this and  our remarks at the end of the introduction we may write
\[f=\sum_{i=1}^n a_i\chi_{A_i} \]
where $0<a_1<a_2<\ldots<a_n$,  and $A_i$ are disjoint sets that  have   positive finite measure. 
 
Assume to the contrary that there exist an interval $(x-r,x+r)$ of positive radius  $r$ in which  
     the frequency function is never zero. Since $\mathcal{M}f$ is lower semicontinuous, there exist a sequence of points $\{x_k\}_{k\in \mathbb{N}}\subset (x-r,x+r)$ converging to $x$ with $\mathcal{M}f(x_k)\geq \mathcal{M}f(x)+\epsilon$.

 Let $z\in (x-r,x+r)$ be a point with  $\mathcal{A}_sf(z)\rightarrow f(z)$ as $s\rightarrow 0$. If   $f(z)=a_n$,  that $a_n$ is the greatest value $f$ can attain leads first to the equality   $\mathcal{A}_{\mathcal{T}f(z)}f(z)=\mathcal{M}f(z)=a_n$, and then  further  to the conlusion that within   $(z-\mathcal{T}f(z),z+\mathcal{T}f(z))$ the function $f$ must be $a_n$ almost everywhere. But this   implies $\mathcal{T}f(z)=0$, a contradiction.  Thus $|A_n\cap (x-r,x+r)|=0.$ 

We must therefore have $\mathcal{M}f(x)<a_{n-1}$, for otherwise the  elements of our sequence $\{x_k\}_{k\in \mathbb{N}}$   would satisfy   $\mathcal{M}f(x_k)\geq a_{n-1}+\epsilon$, which implies,  by the conclusion of the last paragraph, for $k$ large enough $\mathcal{T}f(x_k)\geq 2r/3$, and this contradicts Lemma 3. 

If $|A_{n-1}\cap (x-r/3^i,x+r/3^i) |>0$ for every $i\in \mathbb{N}$,  we can extract a point $z_i$   from each of these sets satisfying $\mathcal{A}_sf(z_i)\rightarrow f(z_i)=a_{n-1}$ as $s\rightarrow 0.$  Therefore  $\mathcal{A}_{\mathcal{T}f(z_i)}f(z_i)=\mathcal{M}f(z_i)\geq a_{n-1}$.  If for a natural number $i$ we have  $\mathcal{T}f(z_i)\leq 2r/3$, then     $f=a_{n-1}$ almost everywhere in $ (z_i-\mathcal{T}f(z_i),z_i+\mathcal{T}f(z_i))$, and this contradicts our assumption that the frequency function  is never zero in $(x-r,x+r)$. Therefore $\mathcal{T}f(z_i)> 2r/3$ for all $i \in \mathbb{N}.$ But this in its turn contradicts Lemma 3. Hence there must be a natural number $i_1$ for which   $|A_{n-1}\cap (x-r/3^{i_1},x+r/3^{i_1}) |=0$. 

We now repeat the  arguments of the last two paragraphs for each $1<m<n$ to first show that $\mathcal{M}f(x)<a_{n-m}$, and then that $|A_{n-m}\cap (x-r/3^{i_{m}},x+r/3^{i_{m}})|=0$ with $i_1<i_2\ldots<i_{n-1}$. Then for large enough $k$ the elements of the sequence $\{x_k\}_{k\in \mathbb{N}}$ must satisfy $\mathcal{T}f(x_k)>r/3^{k_{n-1}+1}$, and this contradicts  Lemma 3. Thus our assumption that the frequency function is never zero in $(x-r,x+r)$ must be wrong.

\end{proof}

\subsection{Proof of Theorem 6} For  the existence of non-Lebesgue points we rely on Lemma 3 and  the existence of exceptional points. We will assume the existence of a radius $r>0$ for which all points in  $(x-r,x+r)$ are Lebesgue points, and proceed to obtain a contradiction by locating an exceptional point in this interval, and proving that it cannot be a Lebesgue point. For the existence of discontinuities of $\mathcal{M}f$ we assume the existence of a radius $r>0$ for which  $\mathcal{M}f$ is continuous  in  $(x-r,x+r)$, and use topological arguments. After the proof we give an example showing that it is not possible to extend this theorem to non-simple functions. Then along with a heuristic explanation we give another  example making clear that it is not possible to obtain a full converse in this theorem. Finally, to  show the impossibility of improving upon   this  theorem  in  another direction,  we provide two more examples demonstrating that 
even for characteristic functions a  point of discontinuity of $\mathcal{M}f$   may well be a Lebesgue point of $f$, and at a non-Lebesgue point of $f$, we may  have $\mathcal{M}f$   continuous.

\begin{proof}
If $f$ is zero almost everywhere, then $\mathcal{M}f$ is never discontinuous,  therefore we may  assume otherwise.  Owing to this and  our remarks at the end of the introduction we may write
\[f=\sum_{i=0}^n a_i\chi_{A_i} \]
where $0=a_0<a_1<a_2<\ldots<a_n$,  and $A_i$ are disjoint sets that when $i>0$ have    finite positive  measure.  We let $b$ to be the minimum distance between any two of these coefficients. Since $\mathcal{M}f$ is discontinuous at $x$, there exist an $\epsilon >0$ and a sequence $\{x_k\}_{k\in \mathbb{N}}$  converging to $x$ with $\mathcal{M}f(x_k)\geq \mathcal{M}f(x)+\epsilon.$

We observe that if  $y$ is a Lebesgue point for this function, that is if   
\[
\lim_{s\rightarrow 0} \frac{1}{2s}\int_{y-s}^{y+s}|f(t)-c|dt=0, 
\]
then $c=a_j$ for some $j.$  For we have  
\begin{equation*}
\begin{aligned}
\frac{1}{2s}\int_{y-s}^{y+s}|f(t)-c|dt&=\sum_{i=0}^n|a_i-c|\frac{|A_i \cap[y-s,y+s]|}{2s}
&\geq \min_{i}|a_i-c| ,
\end{aligned}
\end{equation*}
from which our observation is immediate. We further observe from this that $y$   is a point of density for  $A_j$ if $c=a_j$.

We assume to the contrary that there exists  a radius $r>0$ for which all points in  $(x-r,x+r)$ are Lebesgue points of $f$. Thus in particular  $x$ is a Lebesgue point, which as we observed above means it is a point of density for $A^j$   for some $j$. Therefore $\mathcal{M}f(x)\geq a_j$, and   $|A_j\cap (x-s,x+s)|>0$ for any positive $s$. If we  have $|A_j\cap (x-s,x+s)|=2s$ some $s$,  then $\mathcal{M}f(x_k)\geq a_j+\epsilon $ implies that for $k$ large enough  $\mathcal{T}f(x_k)\geq s/2$, and this contradicts Lemma 3. Therefore   $0<|A_j\cap (x-s,x+s)|<2s$ for any positive $s$.

As $x$ is  a point of density for $A_j$ we can find a radius $r'<r$, such that $|A_j\cap (x-r',x+r')|>3r'/2$. Therefore $A_j\cap (x-r',x)$ and $A_j\cap (x,x+r')$ both have measure at least $r'/2$, and thus both of them have density points, let $u\in A_j\cap (x-r',x)$ and $v\in A_j\cap (x,x+r')$ be such points. The  set $A_j^c\cap (u,v)$ have positive measure by the conclusion of the last paragraph. Therefore it has an exceptional point $y$, and as $(-\infty,u]\cup [v,\infty) $ are density points for $(A_j^c\cap (u,v))^c$  we must have $y\in (u,v)$. Owing to  this $y$ is also an exceptional point for $A_j$.

But as we assumed all points in $(x-r,x+r)$ to be Lebesgue points for $f$, this $y$ must also be a Lebesgue point, from which it follows that it is a point of density for some $A_l$. Clearly $l=j$ is not possible. To see that  $l\neq j$ is also not possible, observe that   in this case  $A_l\subseteq A_j^c$, therefore $y$ is  a point of density for $A_j^c$. Thus our assumption must be wrong, and $(x-r,x+r)$  contains a non-Lebesgue point.

We now turn our attention to the other direction. If $f$ is zero almost everywhere, then all points are Lebesgue, therefore we may  write $f=\chi_A$ where $0<|A|<\infty.$ 

Assume to the contrary that there exist an $r>0$ such that $\mathcal{M}f$ is continuous in $(x-r,x+r).$  We observe that  as $x$  is a non-Lebesgue point $0<|A\cap (x-s,x+s)|<2s$ for every $s>0$. Let $U=\{z\in \mathbb{R}:\mathcal{M}f(z)=1  \}$. Clearly if $z$ is a point of density for $A$, then it is in $U$. On the other hand, if $z$ is a point of density for $A^c$, it cannot be in $U.$ Hence we have $|(A\setminus U)\cup (U\setminus A)|=0$. As we assumed $\mathcal{M}f$ to be continous on $(x-r,x+r),$ the set  $(x-r,x+r)\cap U^c$ is open. It can neither be  empty, nor all of $(x-r,x+r)$, for this would contradict   $0<|A\cap (x-r,x+r)|<2r$. Therefore it is the union of an at most countable collection of disjoint open intervals,  one $(a,b)$ of which is   such that either $a\neq x-r$ or $b\neq x+r$. We let $y$ to be the endpoint for which this is true. This means $\mathcal{M}f(y)<1$. But as $y\in U$  this is a contradiction.

\end{proof}

 We now  show that Theorem 6 is not valid for non-simple functions. We let $\phi(x):=(-|x|+1)\chi_{[-1,1]}(x),$ and 
 \[f_6(x):=\sum_{k=1}^{\infty}\phi(2^{4k}x-2^{3k}).  \]
 This function is a sequence of isosceles triangles of height 1 and base length   $2^{-4k+1}$. It  is continuous everywhere except at the origin. Since continuity at $x$ implies that $x$ is a Lebesgue point,   all points in $\mathbb{R}-\{0\}$ must be Lebesgue points. Furthermore, as the triangles get thinner very fast, the origin is also a Lebesgue point. But   clearly $\mathcal{M}f_6$ is discontinuous at the origin. 

The maximal function is calculated by taking supremum over averages of all positive radii, and is  of global nature, whereas the  Lebesgue points are determined via a limit of averages of radii converging to zero, and   is local.  Obtaining information regarding global phenomena from local phenomena is of course much  harder than doing the converse, and most of the time impossible. A full converse in our theorem is not possible exactly  due  to this reason, as the following example makes clear. Let 
\[f_7(x):=\chi_{(-1,0)}+100 \chi_{(1,2)}.  \]
Clearly $0$ is a non-Lebesgue point of $f_7$, but $\mathcal{M}f_7$ is continuous around $0$.
But when we restrict ourselves to  characteristic functions we obtain  global control over values the function can take, and this allowed  us to prove the partial converse. 
 
Let us consider the following function
\[f_8(x):=\sum_{k\in \mathbb{N}}\chi_{(2^{-k}-2^{-2k-1},2^{-k})}(x).  \] 
 Heuristically this  amounts to considering  the  intervals $(2^{-k-1},2^{-k}),$ dividing them into $2^{k}$ pieces and taking the rightmost piece. Therefore from each dyadic interval we are taking less and less, and   this makes $0$ a Lebesgue point. But clearly $\mathcal{M}f_8$ is discontinuous at $0$.

The last example is more interesting. Let $a_0=0, a_1=1$ and then $a_{k+1}:=a_k+2^{-k(k+1)/2}$. This sequence clearly converges to a limit $a<2$.  Let  $b_k:=(a_k+a_{k+1})/2$, and define
\[f_9(x):=\chi_{(a,a+1)}(x)+\sum_{k=0}^{\infty}\chi_{(b_k,a_{k+1})}(x).\]  
Essentially this means  taking the right halves of the intervals $(a_k,a_{k+1})$, the lengths of which decrease at an ever increasing pace.   We have $\mathcal{A}_{a-a_k}f_9(a)=3/4$ for each $k$,   showing that $a$ is not a Lebesgue point, while the averages $\mathcal{A}_{a-b_k}f_9(a)$ increase to $1$ as $k$ increases, implying $\mathcal{M}f_9(a)=1$. This, by lower semi-continuity of $\mathcal{M}f_9$  and $f_9$ being a characteristic function, implies the continuity of  $\mathcal{M}f_9$  at the point $a$.

\end{document}